\documentclass[12pt]{amsart}
\usepackage{amsmath, amssymb, latexsym}
\usepackage{graphicx}

\usepackage{tikz}
\usepackage{pgfbaselayers}
\usepackage[colorlinks=true, pdfstartview=FitV, linkcolor=blue, citecolor=blue, urlcolor=blue]{hyperref}

\definecolor{rojo}{rgb}{1,0,0}
\definecolor{blanco}{rgb}{1,1,1}

\allowdisplaybreaks

\newcommand{\ra}{\rightarrow}

\newcommand{\bey}{\begin{eqnarray*}}
\newcommand{\eey}{\end{eqnarray*}}
\newcommand{\ba}{\begin{align}}
\newcommand{\ea}{\end{align}}
\newcommand{\bea}{\begin{align*}}
\newcommand{\eea}{\end{align*}}
\newcommand{\be}{\begin{equation}}
\newcommand{\ee}{\end{equation}}
\newcommand{\R}{\mathbb R}
\newcommand{\Z}{\mathbb Z}

\newcommand{\D}{\mathcal D}

\newcommand{\bc}{\begin{center}}
\newcommand{\ec}{\end{center}}

\newcommand{\intRn}{\int_{\R^n}}

\newcommand{\al}{\alpha}

\newcommand{\cz}{Calder\'on-Zygmund\,}
\newcommand{\de}{\delta}

\newtheorem{thm}{Theorem}[section]
\newtheorem{cor}[thm]{Corollary}
\newtheorem{lem}[thm]{Lemma}

\theoremstyle{definition}

\newtheorem{remk}[thm]{Remark}

\numberwithin{equation}{section}

\newcommand{\comment}[1]{\vskip.3cm
\fbox{%
\parbox{0.93\linewidth}{\footnotesize #1}}
\vskip.3cm}

\begin{document}

\title [Sharp Weighted Bounds]
{Sharp weighted bounds for fractional integral operators}

\author[M. T. Lacey]{Michael T.  Lacey}
\author[K. Moen]{Kabe Moen}
\author[C. P\'erez]{Carlos P\'erez}
\author[R. H. Torres]{Rodolfo H. Torres}

\address{Michael T. Lacey\\
Department of Mathematics, Georgia Institute of Technology, Atlanta GA 30332, USA.}
\email{lacey@math.gatech.edu }

\address{Kabe Moen\\
Department of Mathematics, University of Kansas, 405 Snow Hall 1460
Jayhawk Blvd, Lawrence, Kansas 66045-7523, USA.}
\email{moen@math.ku.edu}

\address{Carlos P\'erez\\
Departamento De An\'alisis Matem\'atico, Facultad de Matem\'aticas,
Universidad De Sevilla, 41080 Sevilla, Spain.}
\email{carlosperez@us.es}

\address{Rodolfo H. Torres\\
Department of Mathematics, University of Kansas, 405 Snow Hall 1460
Jayhawk Blvd, Lawrence, Kansas 66045-7523, USA.}
\email{torres@math.ku.edu}

\begin{abstract}
The relationship between the  operator norms of fractional integral
operators acting on weighted Lebesgue spaces and the constant of the weights is investigated. Sharp bounds
are obtained for both the fractional integral operators and the
associated fractional maximal functions.  As an application improved Sobolev inequalities are obtained.
Some of the techniques used include  a sharp off-diagonal  version
of the extrapolation theorem of Rubio de Francia and characterizations of two-weight norm inequalities.

\end{abstract}


\maketitle

\section{Introduction} Recall that a non-negative locally integrable function, or weight, $w$ is said to belong to the $A_p$
class for $1<p<\infty$  if it satisfies the condition
$$[w]_{A_p}\equiv\sup_Q\left( \frac{1}{|Q|}\int_Q w(x) \ dx\right)\left(\frac{1}{|Q|}\int_Q w(x)^{1-p'} \ dx\right)^{p-1}<\infty,$$
where $p'$ is the dual exponent of $p$ defined by the equation
$1/p+1/p'=1$.  Muckenhoupt \cite{Mu} showed that the weights
satisfying the $A_p$ condition are exactly the weights for which the
Hardy-Littlewood maximal function
$$
Mf(x)=\sup_{Q\ni x} \frac{1}{|Q|}\int_Q |f(y)| \ dy
$$
is bounded on $L^p(w)$.  Hunt, Muckenhoupt, and Wheeden \cite{HMW} extended the weighted  theory to the study of the
Hilbert transform
$$
\mathcal H f(x)=p.v. \int_{\R} \frac{f(y)}{x-y}\, dy.
$$
They showed that the
$A_p$ condition also characterizes the $L^p(w)$ boundedness of this
operator. Coifman and Fefferman \cite{CF} extended the $A_p$
theory to general Calder\'on-Zygmund operators. For example, to
 operators that are bounded, say on $L^2(\R^n)$, and of the form %
$$
Tf(x)=p.v. \int_{\R^n} f(y)K(x,y)\, dy,
$$
where
$$
|\partial^\beta K(x,y)| \leq c |x-y|^{-n-|\beta|}.
$$

Bounds on the operators norms in terms of the $A_p$ constants of the weights have been investigated as well. Buckley \cite{Buck} showed that for $1<p<\infty$,  $M$
satisfies
\begin{equation}\label{sharpmaximal}
\|M\|_{L^p(w) \to L^p(w)} \leq c \, [w]_{A_p}^{1/(p-1)}
\end{equation}
and the exponent $1/(p-1)$ is the best possible.   A new  and rather
simple proof of both Muckenhoupt's and Buckley's results was
recently given by Lerner \cite{Ler}.
The weak-type bound also observed by Buckley \cite{Buck} is
\be \label{weakannonimous}
\|M\|_{L^p(w) \to L^{p,\infty}(w)} \leq c\,  [w]_{A_p}^{1/p}.
\ee

For singular integrals operators, however, only partial results are
known. The interest in sharp weighted norm for singular integral
operators is motivated in part by applications in partial
differential equations. We refer the reader to Astala, Iwaniec, and
Saksman \cite{AsIwSa}; and Petermichl and Volberg \cite{PetVol} for
such applications.    Petermichl \cite{Pet1},  \cite{Pet2} showed
that
\begin{equation}\label{sharpcz} \|T\|_{L^p(w) \to L^p(w)} \leq c \, [w]_{A_{p}}^{\max\{1,1/(p-1)\}},\end{equation}
where $T$ is either the Hilbert or  one of the Riesz transforms in $\R^n$,
$$
R_jf(x)=c_n\, p.v.  \int_{\R^n} \frac{x_j-y_j}{|x-y|^{n+1}} f(y)\, dy.
$$
Petermichl's results were obtained for $p=2$ using Bellman function
methods. The general case $p\neq2$ then follows by the sharp version
of the Rubio de Francia extrapolation theorem given by
Dragi\u{c}evi\'c, Grafakos, Pereyra, and Petermichl \cite{DGPP}. We
recall that the original proof of the extrapolation theorem was given by Rubio de Francia in
\cite{Rub} and it was not constructive. Garc\'ia-Cuerva then gave a constructive
proof that can be found in \cite[p.434]{GCRdF} and which has been used
to get the sharp version in \cite{DGPP}. It is important to remark that so far no proof of the $L^p$ version of Petermilch's result is know without invoking extrapolation.
These are the best known results so far and whether \eqref{sharpcz}
holds for general \cz operators is not known.

There are also other estimates for \cz operators involving weights
which have received attention over the years. In particular, there
is  the ``Muckenhoupt-Wheeden conjecture"
\be \label{mwconjecture} \|Tf\|_{L^{1,\infty}(w) }\leq c\, \|f
\|_{L^1(Mw)}, \ee
for arbitrary weight $w$, and the ``linear growth conjecture" for
$1<p<\infty$, \be \label{lgconjecture} \|T\|_{L^p(w)\ra
L^{p,\infty}(w)} \le c_p [w]_{A_p}. \ee Both these conjectures
remain very difficult open problems. Some progress has been recently
made by Lerner, Ombrosi and P\'erez \cite{LOP1}, \cite{LOP2}.

Motivated by all these estimates, we investigate in this article the sharp weighted bounds for fractional integral operators and the related maximal functions.

For $0<\al<n$, the fractional integral operator or Riesz
potential  $I_\alpha$ is defined by
$$I_\al f(x)=\intRn \frac{f(y)}{|x-y|^{n-\al}}dy, $$
while the related fractional maximal operator  $M_\alpha$ is given by
$$M_\al f(x)=\sup_{Q\ni x} \frac{1}{|Q|^{1-\al/n}}\int_Q |f(y)| \ dy.$$
These operators play an important role in analysis, particularly in
the study of differentiability or smoothness properties a functions.
See the books by Stein \cite{St} or Grafakos \cite{GrCF}  for the
basic properties of these operators.

Weighted inequalities for these operators and more general potential
operators have been studied in depth. See e.g.  the works of
Muckenhoupt and Wheeden \cite{MW}, Sawyer \cite{Saw11}, \cite{Saw2},
Gabidzashvili and Kokilashvili \cite{GK}, Sawyer and Wheeden \cite{SW}, and P\'erez
\cite{CP2}, \cite{CP3}.
Such estimates naturally appear in problems in partial differential
equations and quantum mechanics.

In \cite{MW}, the authors  characterized the weighted
strong-type inequality for fractional operators in terms of the so-called $ A_{p,q}$ condition.
For $1<p<n/\al$ and $q$  defined by $1/q=1/p-\al/n$, they showed that for all  $ f\geq 0$,
\begin{equation} \label{FrcInt}
\left(\intRn (wT_\al f )^q \ dx\right)^{1/q}\leq c\left(\intRn (w f)^p \ dx\right)^{1/p},
\end{equation}
where $T_\al=I_\al$ or $M_\al$,   if
and only if $w\in A_{p,q}$. That is,
$$ [w]_{A_{p,q}}\equiv \sup_Q\left(\frac{1}{|Q|}\int_Q w^q \ dx\right)\left(\frac{1}{|Q|}\int_Q w^{-p'}\ dx\right)^{q/p'}<\infty. $$

The connection between the $A_{p,q}$ constant $[w]_{A_{p,q}}$ and
the operator norms of these fractional operators is the main focus
of this article. We will obtain the analogous estimates of \eqref{sharpmaximal}, \eqref{weakannonimous},
\eqref{sharpcz}, \eqref{mwconjecture}, and \eqref{lgconjecture} in
the fractional integral case.

At a formal level, the case $\al =0$ corresponds to the \cz case where, as mentioned, some estimates have not been obtained yet. Though  for $\al>0$ one deals with positive operators,  the corresponding estimates still remain difficult to be proved and we need to use a set of tools different from the ones used in the \cz situation.

Our main result, Theorem \ref{fullrange} below, is the sharp bound
$$\|I_\al\|_{L^p(w^p)\ra L^q(w^q)}\leq c[w]_{A_{p,q}}^{(1-\frac{\al}{n})\max\{1,\frac{p'}{q}\}}.$$
This is the analogous estimate of \eqref{sharpcz} for fractional integral operators.\\

{\bf Acknowledgements}

First author's research supported in part by National Science Foundation under grant DMS 0456611. Second  and fourth authors' research supported in part by the
National Science Foundation under grant DMS 0800492. Third author's
research supported in part by the Spanish Ministry of Science under
research grant MTM2006-05622. Part of the research leading to the results presented in this article was conducted when C. P\'erez visited  the University of Kansas, Lawrence  during the academic year 2007-2008
 and in the spring of 2009.  Finally, the  authors are very grateful to the ``Centre de Recerca
Matem\`{a}tica"  for the invitation to participate in a special research
programme in Analysis, held in the spring of 2009 where this project was finished.

\section{Description of the main results}
We start by observing that to obtain sharp bounds  for the
strong-type inequalities for $I_\alpha$ it is enough to obtain sharp
bounds for the weak-type ones. This is due to Sawyer's deep results
on the characterization of two-weight norm inequalities for
$I_\alpha$.  In fact, he proved in \cite{Saw2} that for two positive
locally integrable function $v$ and $u$,   and $1<p\leq q <\infty$,
$$I_\alpha: L^p(v)\to L^q(u)$$
if and only if  $u$ and the function $\sigma= v^{1-p'}$ satisfy the
{\it (local) testing conditions}
$$[u,\sigma]_{S_{p,q}}\equiv \sup_Q \sigma(Q)^{-1/p} \|\chi_QI_\alpha(\chi_Q \sigma)\|_{L^q(u)}<\infty$$
and
$$[\sigma, u ]_{S_{q',p'}}\equiv \sup_Q u(Q)^{-1/q'} \|\chi_QI_\alpha(\chi_Q u)\|_{L^{p'}(\sigma)}<\infty.$$
Moreover, his proof shows that actually
\begin{equation}\label{strong2}
\|I_\alpha\|_{L^p(v)\to L^q(u)} \approx [u,\sigma]_{S_{p,q}} + [\sigma, u ]_{S_{q',p'}}.
\end{equation}
On the other hand in his characterization of the weak-type,
two-weight inequalities for $I_\alpha$, Sawyer \cite{Saw11} also
showed that
%
$$
\|I_\alpha\|_{L^p(v)\to L^{q,\infty}(u)} \approx [\sigma, u ]_{S_{q',p'}}.$$
%
Combining  \eqref{strong2} and \eqref{weak2} it follows that
\begin{equation}\label{strongweak2}
\|I_\alpha\|_{L^p(v)\to L^q(u)} \approx \|I_\alpha\|_{L^{q'}(u^{1-q'})\to L^{p',\infty}(v^{1-p'})} +\|I_\alpha\|_{L^p(v)\to L^{q,\infty}(u)}.
\end{equation}
If we now set $u=w^q$ and $v=w^p$, we finally obtain the one-weight
estimate
\begin{equation}\label{strongweak}
\|I_\alpha\|_{L^p(w^p)\to L^q(w^q)} \approx
\|I_\alpha\|_{L^{q'}(w^{-q'})\to L^{p',\infty}(w^{-p'})}
+\|I_\alpha\|_{L^p(w^p)\to L^{q,\infty}(w^q)}.
\end{equation}

We will obtain sharp bounds for the weak-type norms in the right hand side of \eqref{strongweak} in two different ways, each of which is of interest on its own. Our first approach is based on an off-diagonal extrapolation theorem by
Harboure, Mac\'ias, and Segovia \cite{HMS}.  A second one is based in yet another characterization of  two-weight norm inequalities for $I_\alpha$  in the case $p<q$,  in terms of certain {\it (global) testing condition} and which is due to  Gabidzashvili and Kokilashvili \cite{GK}.

We present now the extrapolation results.
The proof follows the original one, except that we carefully track  the dependence of the estimates in terms of the $A_{p,q}$ constants of the weights.
\begin{thm} \label{ExtThm} Suppose that $T$ is an operator defined on an appropriate class of functions, (e.g. $C_c^\infty$, or $\bigcup_p L^p(w^p)$).  Suppose further that  $p_0$ and $q_0$ are exponents with $1\leq p_0\leq q_0<\infty$, and such that
$$\|wTf \|_{L^{q_0}(\R^n)}\leq c[w]_{A_{p_0,q_0}}^\gamma \|wf\|_{L^{p_0}(\R^n)}$$
holds for all $w\in A_{p_0,q_0}$ and some $\gamma>0$.  Then,
 $$\|wTf \|_{L^{q}(\R^n)}\leq c[w]_{A_{p,q}}^{\gamma \max\{1,\frac{q_0}{p_0'}\frac{p'}{q'}\}} \|wf \|_{L^{p}(\R^n)}$$
holds for all $p$ and $q$ satisfying $1<p\leq q<\infty$ and
$$\frac{1}{p}-\frac{1}{q}=\frac{1}{p_0}-\frac{1}{q_0},$$
and all weight $w\in A_{p,q}$.
\end{thm}

As a consequence we have the following weak extrapolation theorem
using an idea from Grafakos and Martell \cite{GM}.

\begin{cor}\label{ExtWeak} Suppose that for some $1\leq p_0\leq q_0<\infty$, an operator $T$ satisfies the weak-type $(p_0,q_0)$ inequality
$$ \|Tf\|_{L^{q_0,\infty}(w^{q_0})}\leq c [w]^\gamma_{A_{p_0,q_0}} \|wf\|_{L^{p_0}(\R^n)}$$
 for every $w\in A_{p_0,q_0}$ and some $\gamma>0$.  Then $T$ also satisfies the weak-type $(p,q)$ inequality,
$$ \|Tf\|_{L^{q,\infty}(w^{q})}\leq c[w]_{A_{p,q}}^{\gamma \max\{1,\frac{q_0}{p_0'}\frac{p'}{q}\}} \|wf\|_{L^{p}(\R^n)}$$
for all $1<p\leq q<\infty$ that satisfy
$$\frac{1}{p}-\frac{1}{q}=\frac{1}{p_0}-\frac{1}{q_0}$$
and all $w\in A_{p,q}$.
\end{cor}

We will use the above corollary to obtain sharp weak bounds in the whole range of exponents for $I_\alpha$. As already described, this leads to strong-type estimates too. Nevertheless, for a certain range of exponents the strong-type estimates can be obtained in a more direct way without relying on the difficult two-weight results. 

It is not obvious a priori what the analogous of \eqref{sharpcz} should be for $I_\alpha$. A possible guess is
\begin{equation}\label{guess1}
 \|w\,I_\al f\|_{ L^{q}(\mathbb{R}^{n})} \leq
c\,[w]_{A_{p,q}}^{ \max\{1,\frac{p'}{q} \} }\,
\|w\,f\|_{L^{p}(\mathbb{R}^{n})}.
\end{equation}
Note that formally, the estimate reduces to  \eqref{sharpcz} when
$\alpha=0$ suggesting it could be sharp. While it is possible to
obtain such estimate, simple examples indicate it is not the best one.
In fact, we will show in this article a direct proof of the
following estimate.

\begin{thm} \label{ShpFracInt} Let $1<p_0<n/\al$ and $q_0$ be defined by the
equations $1/q_0=1/p_0-\al/n$ and $q_0/p'_0=1-\alpha/n$, and  let $w\in A_{p_0,q_0}$.
Then,
\begin{equation}\label{starting1}
 \|wI_\al f \|_{L^{q_0}(\R^n)} \leq
c\,[w]_{A_{p_0,q_0}}\|wf\|_{L^{p_0}(\R^n)}.
\end{equation}

\end{thm}

The pair $(p_0,q_0)$ in the above theorem could be seen as the replacement of the $L^2$ case when $\alpha=0$. That is, it yields a linear growth on the weight constant. However, unlike the case $\alpha=0$, one can check that starting from this point $(p_0,q_0)$, extrapolation and duality give sharp estimates for a reduced set of exponents. See \eqref{reducedset} below. To get the full range we use first Corollary \ref{ExtWeak} to obtain sharp estimates for the weak-type $(p,q)$ inequality for $I_\al$.  We
have the following result.

\begin{thm}\label{ShpWk} Suppose that $1\leq p<n/\al$ and that $q$ satisfies $1/q=1/p-\al/n$.  Then
\be \label{weakest} \|I_\al f\|_{L^{q,\infty}(w^q)}\leq
c\,[w]_{A_{p,q}}^{1-\frac{\al}{n}} \|w\, f\|_{L^p(\R^n)}.\ee
Furthermore, the exponent $1-\frac{\al}{n}$ is sharp.
\end{thm}
We will also present a second proof of Theorem \ref{ShpWk} for $p>1$ without using extrapolation.

\begin{remk}
Once again,  the estimate in the above weak-type result should be contrasted with the case  $\al=0$ and the  linear growth conjecture  for a \cz operator $T$. Namely,
$$
\|T\|_{L^p(w)\ra L^{p,\infty}(w)} \le c_p [w]_{A_p}.
$$
Such results have remained elusive so far. For the best available
result see \cite{LOP2}.

The extrapolation proof of Theorem \ref{ShpWk}  will also show that for any weight $u$ the weak-type inequality
$$\|I_\al f\|_{L^{(n/\al)', \infty}(u)}\leq c\, \|f\|_{L^1((Mu)^{1-\frac{\al}{n}})}$$
holds.  For $\al=0$ the analogous version of this inequality is the
Muckenhoupt-Wheeden conjecture
$$\|Tf\|_{L^{1,\infty}(w) }\leq c\, \|f \|_{L^1(Mw)},$$
which is an open problem.
\end{remk}

As a consequence of the weak-type estimate \eqref{weakest}  we obtain the sharp bounds indicated by examples.
\begin{thm} \label{fullrange} Let $1<p<n/\al$ and $q$ be defined by the equation $1/q=1/p-\al/n$, and  let $w\in A_{p,q}$. Then,
\be \label{shpconjagain} \|I_\al\|_{L^p(w^p) \to L^q(w^q)}\leq
c\,[w]_{A_{p,q}}^{(1-\frac{\al}{n})\max\{1,\frac{p'}{q}\}}. \ee
Furthermore this estimate is sharp.
\end{thm}

Another consequence of   \eqref{weakest}  is  a Sobolev-type estimate.  We obtain this when we use the fact
that weak-type inequalities implies strong-type inequalities when a gradient
operator is involved.  We have
the following result based on the ideas of Long and Nie \cite{LN}.  See also Hajlasz \cite{Haj}.

\begin{thm} \label{weakstrong} Let $p\geq 1$ and let $w\in A_{p,q}$ with $q$ satisfying $1/p-1/q=1/n$. Then,  for any Lipschitz function $f$ with compact support,
\be\label{gradient}
 \left(\intRn (|f(x)|w(x))^q \ dx \right)^{1/q} \leq c\,[w]_{A_{p,q}}^{1/n'} \left(\intRn (|\nabla f(x)| w(x))^p \ dx \right)^{1/p}.
 \ee
\end{thm}

\begin{remk} We note that this estimate is better than what the strong bound on
$I_1$ in Theorem~\ref{fullrange}  gives.  In fact, for $f$  sufficiently smooth and compactly
supported,  we have the estimate
$$|f(x)|\leq c I_1(|\nabla f|)(x).$$
Hence,  if we applied Theorem \ref{fullrange} we obtain the estimate
$$\|f w\|_{L^q}\leq c\,[w]_{A_{p,q}}^{ 1/n' \max\{1,p'/q\} }\|\nabla f w\|_{L^p}.$$
However, Theorem \ref{weakstrong} gives a better growth in terms of the weight, simply
$[w]_{A_{p,q}}^{1/n'}$. This is a better  growth in the range $1<p<\min(2n', n)$ (i.e. $p'/q>1$) where the  estimate  \eqref{shpconjagain} only  gives
$[w]_{A_{p,q}}^{p'/(qn')}$. Note also that \eqref{gradient} includes the case $p=1$,  which cannot be obtained using Theorem~\ref{fullrange}.
\end{remk}

We also find the sharp constant for $M_\al$ in the full range of exponents.
\begin{thm}\label{ShpFracMax}  Suppose $0\leq \al <n$, $1<p<n/\al$ and $q$ is defined by the relationship $1/q=1/p-\al/n$.  If $w\in A_{p,q}$,  then
\be \label{FMest} \|wM_\al f \|_{L^q} \leq c
[w]_{A_{p,q}}^{\frac{p'}{q}(1-\frac{\al}{n})}\|wf \|_{L^p}. \ee
Furthermore, the exponent $\frac{p'}{q}(1-\frac{\al}{n})$ is sharp.
\end{thm}

Note one more time that formally replacing $\alpha=0$ the estimates clearly generalize the result in \cite{Buck}.

\begin{remk}
 We also note that there is a weak-type estimate for $M_\al$.  For $p\geq 1$ and $1/q=1/p-\al/n$, standard covering methods
give
\be \label{weakMximal}
\|M_\al\|_{ L^p(w^p) \to L^{q,\infty}(w^q)} \leq  c\,   [w]_{A_{p,q}}^{1/q}.
\ee
See for instance the book by Garcia-Cuerva
and Rubio de Francia \cite[pp. 387-393]{GCRdF}, for the estimate in the case $\al=0$.
\end{remk}

\begin{remk}
Continuing with the formal comparison with the case $\alpha=0$, it would be interesting  to know if the
analog of \eqref{strongweak} also holds for Calder\'on-Zygmund singular integrals. Namely,
\begin{equation}\label{strongweakT}
\|T\|_{L^p(w)\to L^p(w)} \approx \|T^*\|_{L^{p'}(w^{1-p'})\to L^{p',\infty}(w^{1-p'})} +\|T\|_{L^p(w)\to L^{p,\infty}(w)}.
\end{equation}
This estimate, if true, may be beyond reach with the current available techniques.
\end{remk}

The rest of the paper is organized as follows. We separate the proofs of the main results in different sections which are essentially independent of each other.
 In Section~\ref{extrapolation} we
collect some additional definitions  and the proof of the version of the extrapolation result  Theorem~ \ref{ExtThm}.
We repeat  the proof of such result from \cite{HMS}  for the convenience of the reader, but also  to show that the constant we need can
indeed be tracked through  the computations. A faithful reader familiar with the extrapolation result may skip the details, move directly to the following sections of the article, and come back later to Section~\ref{extrapolation} to verify our claims.
 Section~\ref{bounds}  contains the proof of
of Theorem \ref{ShpFracInt}. 
The proof of Corollary~\ref{ExtWeak} and the two proofs of  the weak-type result for $I_\alpha$, Theorem~\ref{ShpWk}, are in Section~\ref{weak-type}. The proof of Theorem~\ref{fullrange} as a corollary of Theorem~\ref{ShpWk} is in this section too.
The proof of  the result for the fractional maximal function, Theorem~\ref{ShpFracMax},  is presented in Section~\ref{maximalresults}.
In  Section~\ref{examples} we present the examples
for the sharpness in Theorems   \ref{ShpWk}, \ref{fullrange}, and \ref{ShpFracMax}.
Finally,  in Section~\ref{application} we present the proof of the application to Sobolev-type inequalities.

\section{Constants in the off-diagonal extrapolation theorem}\label{extrapolation}

 For a Lebesgue measurable set $E$, $|E|$ will denote
its Lebesgue measure and $w(E)=\int_E w(x) \ dx$ will denote its
weighted measure.  We will be working on weighted versions of the
classical $L^p$ spaces, $L^p(w)$, and also  on the weak-type ones, $L^{p,\infty}(w)$, defined in the usual way with the Lebesgue measure $dx$ replaced by the measure $w\,dx$.
Often, however, it will be convenient  to  viewed the weight  not as a measure but as a
multiplier. For example $f\in L^p(w^p)$ if
$$\|fw\|_{L^p} = \left(\intRn (|f(x)|w(x))^p \ dx\right)^{1/p}<\infty.$$
This is more convenient when dealing with the  $A_{p,q}$ condition already defined in the introduction.  Recall, that for
$1<p\leq q<\infty$, we say $w\in A_{p,q}$  if \be
[w]_{A_{p,q}}\equiv \sup_Q\left(\frac{1}{|Q|}\int_Q w^q \
dx\right)\left(\frac{1}{|Q|}\int_Q w^{-p'}\
dx\right)^{q/p'}<\infty.\label{Apq}
\ee
Also, for $1\leq q<\infty$ we
define the class $A_{1,q}$ to be the weights $w$ that satisfy, \be
\label{Aq1} \left(\frac{1}{|Q|}\int_Q w^q \ dx\right)\leq c \inf_Q
w^q.\ee Here $[w]_{A_{1,q}}$ will denote the smallest constant $c$
that satisfies \eqref{Aq1}.  Notice that $w\in A_{p,q}$ if and only
if $w^q\in A_{1+q/p'}$ with
\be \label{property1}[w]_{A_{p,q}}=[w^q]_{A_{1+q/p'}}.\ee
In particular,  $[w]_{A_{p,q}}\geq 1$. We also note for later use that
\be \label{property2} [w^{-1}]_{A_{q',p'}}=[w]_{A_{p,q}}^{p'/q}.\ee
The term cube will always refer to a cube $Q$ in $\R^n$ with sides parallel to the axis. A multiple $rQ$ of a cube is a cube with the same center of $Q$ and side-length $r$ times as large. By $\D$ we denote the collection of all dyadic cubes in $\R^n$. That is, the collection of all cubes with lower-felt corner $2^{-l}m$ and side-length $2^{-l}$ with $l\in \Z$ and $m\in\Z^n$. As usual, $B(x,r)$ will denote the Euclidean ball in $\R^n$ centered at the point $x$ and with radius $r$.

To prove Theorem \ref{ExtThm} we will need the sharp version of the
Rubio de Francia algorithm given by Garc\'ia-Cuerva.  The proof can
be found in the article \cite{DGPP}.
\begin{lem} \label{RubioLem} Suppose that $r>r_0$, $v\in A_r$, and $g$ is a non-negative function in $L^{(r/r_0)'}(v)$.  Then,  there exists a function $G$ such that
\begin{enumerate}
\item $G\geq g$,
\item $\|G\|_{L^{(r/r_0)'}(v)}\leq 2 \|g\|_{L^{(r/r_0)'}(v)}$,
\item $Gv\in A_{r_0}$ with $[Gv]_{A_{r_0}}\leq c\,[v]_{A_r}$.
\end{enumerate}
\end{lem}

\begin{proof}[Proof of Theorem \ref{ExtThm}]

 First suppose
$w\in A_{p,q}$ and $1\leq p_0<p$, which implies $q>q_0$.  Then, \bey
\left(\intRn |Tf|^{q} w^{q}\right)^{1/q} &=&\left(\intRn (|Tf|^{q_0})^{q/q_0} w^{q}\right)^{\frac{q_0}{q}\frac{1}{q_0}} \\
&=&\left(\intRn |Tf|^{q_0} g w^{q}\right)^{\frac{1}{q_0}}
 \eey for
some non-negative $g\in L^{(q/q_0)'}(w^q)$ with
$\|g\|_{L^{(q/q_0)'}(w^q)}=1$.  Now, let $r=1+q/p'$ and
$r_0=1+q_0/p_0'$. Since $p>p_0$ we have $r>r_0$. Furthermore,  by the
relationship
$$\frac{1}{p}-\frac{1}{q}=\frac{1}{p_0}-\frac{1}{q_0},$$
 we have $q/q_0=r/r_0$.  Hence by Lemma \ref{RubioLem} and using that $w^q\in A_r$, there exists $G$ with $G\geq g$, $\|G\|_{L^{(r/r_0)'}(w^q)}\leq 2$, $Gw^q\in A_{r_0}$, and  $[Gw^q]_{A_{r_0}}\leq c\,[w^q]_{A_r}=c\,[w]_{A_{p,q}}$.  Also, since  $Gw^q\in A_{r_0}$ then $(Gw^q)^{1/q_0}\in A_{p_0,q_0}$ since,
\bey
[(Gw^q)^{1/q_0}]_{A_{p_0,q_0}}&=&\sup_Q \left(\frac{1}{|Q|}\int_Q (G^{1/q_0}w^{q/q_0})^{q_0} \right)\left(\frac{1}{|Q|}\int_Q (G^{1/q_0}w^{q/q_0})^{-p_0'} \right)^{q_0/p_0'} \\
&=&  \sup_Q \left(\frac{1}{|Q|}\int_Q Gw^q \right)\left(\frac{1}{|Q|}\int_Q (Gw^q)^{-p_0'/q_0} \right)^{q_0/p_0'} \\
&=&[Gw^q]_{A_{r_0}}.
 \eey
 Then,  we can proceed with \bey
\left(\intRn |Tf|^{q} w^{q}\right)^{1/q} &=& \left(\intRn |Tf|^{q_0} g w^{q}\right)^{\frac{1}{q_0}}\\
&\leq&\left(\intRn |Tf|^{q_0} G w^{q}\right)^{\frac{1}{q_0}}\\
&=& \left(\intRn |Tf|^{q_0} (G^{1/q_0} w^{q/q_0})^{q_0}\right)^{\frac{1}{q_0}}\\
&\leq &c\,[G^{1/q_0} w^{q/q_0}]_{A_{p_0,q_0}}^{\gamma} \left(\intRn |f|^{p_0} (G^{1/q_0} w^{q/q_0})^{p_0}\right)^{\frac{1}{p_0}} \\
&= &c\,[Gw^q]_{A_{r_0}}^{\gamma}  \left(\intRn |f|^{p_0}w^{p_0} G^{p_0/q_0} w^{q/(p/p_0)' }\right)^{\frac{1}{p_0}} \\
&\leq &c\,[w]_{A_{p,q}}^{\gamma}  \left(\intRn |f|^{p}w^{p} \right)^{1/p} \left(\intRn G^{(r/r_0)'} w^q \right)^{(p-p_0)/pp_0}\\
&\leq &c\,[w]_{A_{p,q}}^{\gamma}  \left(\intRn |f|^{p}w^{p} \right)^{1/p},
 \eey
where we have used the relationship
$$\frac{1}{p}-\frac{1}{q}=\frac{1}{p_0}-\frac{1}{q_0}.$$
For the case $1<p<p_0$, and hence $q<q_0$, notice that we can write
$$\left(\intRn |f|^{p}w^{p} \right)^{1/p}=\left(\intRn (|fw^{p'}|^{p_0})^{p/p_0}w^{-p'} \right)^{1/p}.$$
Since $p/p_0<1$,  there exists a function $g\geq 0$ satisfying
$$\intRn g^{p/(p-p_0)}w^{-p'}=1$$
such that
$$\left(\intRn |f|^{p}w^{p} \right)^{1/p}=\left(\intRn |fw^{p'}|^{p_0}gw^{-p'} \right)^{1/p_0},$$
see \cite[p. 335]{GrMF}.  Let $h=g^{-p_0'/p_0}$, $r=1+p'/q$ and
$r_0=1+p_0'/q_0$, so that $r>r_0$.  Notice that
$$\frac{1}{p}-\frac1q=\frac{1}{p_0}-\frac{1}{q_0}$$
implies $r/r_0=p'/p_0'$, which in turn yields
\begin{equation}
\label{Algebra} \frac{p_0'}{p_0}\left(\frac{r}{r_0}\right)'
=\frac{p}{p_0 -p}. \end{equation} Hence,
$$\intRn h^{(r/r_0)'}w^{-p'}=\intRn g^{p/(p-p_0)} w^{-p'}=1.$$
Observe that $w^{-p'}\in A_r$, so by Lemma \ref{RubioLem} we obtain
a function $H$ such that $H\geq h$,
$\|H\|_{L^{(r/r_0)'}(w^{-p'})}\leq 2$, and $Hw^{-p'}\in A_{r_0}$
with $[Hw^{-p'}]_{A_{r_0}}\leq
c\,[w^{-p'}]_{A_r}=c\,[w]_{A_{p,q}}^{p'/q}$ .  Now, for $Hw^{-p'}\in
A_{r_0} $ we claim that $(Hw^{-p'})^{-1/p_0'}\in A_{p_0,q_0}$ with
$[(Hw^{-p'})^{-1/p_0'}]_{A_{p_0,q_0}}=[Hw^{p'}]_{A_{r_0}}^{q_0/p_0'}$.
Indeed,
 \bey
[(Hw^{-p'})^{-1/p_0'}]_{A_{p_0,q_0}}&=&\sup_Q \left(\frac{1}{|Q|}\int_Q (H^{-1/p_0'}w^{p'/p_0'})^{q_0} \right) \!\! \left(\frac{1}{|Q|}\int_Q (H^{-1/p_0'}w^{p'/p_0'})^{-p_0'} \right)^{q_0/p_0'} \\
&=&\sup_Q \left(\frac{1}{|Q|}\int_Q (Hw^{-p'})^{-q_0/p_0'} \right) \!\! \left(\frac{1}{|Q|}\int_Q Hw^{-p'} \right)^{q_0/p_0'}\\
&=& [Hw^{-p'}]_{A_{r_0}}^{q_0/p_0'}.
\eey
Finally expressing $g$ in
terms of $h$ and using \eqref{Algebra}, working backwards we have
\bey
\left(\intRn |f|^{p}w^{p} \right)^{1/p} &= & \left(\intRn |f|^{p_0}h^{-p_0/p_0'}w^{p'(p_0-1)} \right)^{1/p_0}\\
&\geq & \left(\intRn |f|^{p_0}H^{-p_0/p_0'}w^{p'(p_0-1)} \right)^{1/p_0}\\
&=& \frac{[(Hw^{-p'})^{-1/p_0'}]_{A_{p_0,q_0}}^{\gamma} }{[(Hw^{-p'})^{-1/p_0'}]_{A_{p_0,q_0}}^{\gamma} }\left(\intRn |f|^{p_0}(H^{-1/p_0'}w^{p'/p_0'})^{p_0} \right)^{1/p_0}\\
&\geq &\frac{c}{[(Hw^{-p'})^{-1/p_0'}]_{A_{p_0,q_0}}^{\gamma} }\left(\intRn |Tf|^{q_0}(H^{-1/p_0'}w^{p'/p_0'})^{q_0} \right)^{1/q_0}\\
&\geq & \frac{c}{[(Hw^{-p'})^{-1/p_0'}]_{A_{p_0,q_0}}^{\gamma} }\left(\intRn |Tf|^{q}w^q\right)^{1/q}\left(\intRn H^{(r/r_0)'}w^{p'} \right)^{q-q_0/qq_0}\\
&\geq & \frac{c}{[(Hw^{-p'})^{-1/p_0'}]_{A_{p_0,q_0}}^{\gamma} }\left(\intRn
|Tf|^{q}w^q\right)^{1/q}.
 \eey
  In the second to last inequality we have used H\"older's inequality for
exponents less than one, i.e., if $0<s<1$ then
$$\|fg\|_{L^1}\geq \|f\|_{L^s}\|g\|_{L^{s'}},$$
where as usual $s'=s/(s-1)$. See \cite[p. 10]{GrCF} for  more details. Thus we have shown,
$$\left(\intRn |Tf|^{q}w^q\right)^{1/q}\leq c\,[(Hw^{-p'})^{-1/p_0'}]_{A_{p_0,q_0}}^{\gamma} \left(\intRn |f|^{p}w^{p} \right)^{1/p}.$$
From here we have
$$\|T\|\leq c\,[(Hw^{-p'})^{-1/p_0'}]_{A_{p_0,q_0}}^{\gamma} =c\,[Hw^{-p'}]_{A_{r_0}}^{{\gamma} \frac{q_0}{p_0'}}\leq c\,[w^{-p'}]_{A_{1+p'/q}}^{{\gamma}\frac{q_0}{p_0'}}=c\,[w]_{A_{p,q}}^{{\gamma} \frac{q_0}{p_0'}\frac{p'}{q}}.$$
This proves the theorem.
\end{proof}

\section{Proofs of strong-type results using extrapolation}\label{bounds}

We will need to use the following weighted versions of $M_\al$. For $0\leq \al < n$, let
$$M_{\al, \nu}^cf(x)= \sup_{Q_x} \frac{1}{\nu(Q_x)^{1-\al/n}}\int_{Q_x} |f(y)| \ d\nu,$$
where the supremum is over all cubes $Q_x$ with center $x$.
 A dyadic version of $M_{\al}$ was first introduced
by Sawyer in \cite{Saw1}.
This maximal function will be an effective tool in obtained the estimates for $I_\al$. The following lemma will be used in the proofs of Theorems
\ref{ShpFracInt} and \ref{ShpFracMax}.

\begin{lem}\label{WeightFrac} Let $0\leq \al <n$ and $\nu$ be a positive Borel measure. Then,
$$\|M_{\al,\nu}^cf\|_{L^q(\nu)}\leq c\, \|f\|_{L^p(\nu)}$$
for all $1<p\leq q<\infty$ that satisfy $1/p-1/q=\al/n$.  Furthermore,
the constant $c$  is independent of $\nu$ (it depends only on the
dimension and $p$).
\end{lem}
The proof of Lemma \ref{WeightFrac} can be obtained by interpolation.  In fact, the
strong $(n/\al,\infty)$ inequality follows directly from H\"older's inequality, while
 a weak-$(1,(n/\al)')$ estimate is a consequence of the
Besicovich covering lemma.

 \begin{proof}[Proof of Theorem \ref{ShpFracInt}]  The equation $q_0/p_0'=1-\al/n$ along with the fact that $1/p_0-1/q_0=\al/n$ yields
 $$p_0=\frac{2-\al/n}{\al/n-(\al/n)^2+1} \quad {\rm and} \quad q_0=\frac{2-\al/n}{1-\al/n}.$$
 We want to show the linear estimate
\begin{equation}\label{mainestimate} \|wI_\al f\|_{L^{q_0}}\leq c\,[w]_{A_{p_0,q_0}}\|wf\|_{L^{p_0}}.\end{equation}
Notice that  \eqref{mainestimate} is equivalent to
 \be\label{equivalent}
   \|I_\al(f\sigma)\|_{L^{q_0}(u)}\leq c\,[w]_{A_{p_0,q_0}}\|f\|_{L^{p_0}(\sigma)},
   \ee
 where $u=w^{q_0}$ and $\sigma=w^{-p_0'}$.  Moreover, by duality, showing \eqref{equivalent}
is equivalent to prove
\be\label{equivalentdual}
\intRn I_\al (f\sigma)g u \ dx \leq c\,[w]_{A_{p_0,q_0}} \left(\intRn f^{p_0}\sigma \ dx\right)^{1/p_0}\left( \intRn g^{q_0'} u \ dx \right)^{1/q_0'}
\ee
for all $f$ and $g$ non-negative bounded functions with compact support.

We first discretize the operator $I_\al$ as follows.  Given a
non-negative function $f$, \bey
I_\al f(x)&=&\sum_{k\in \Z}\int_{2^{k-1}<|x-y|\leq 2^{k}} \frac{f(y)}{|x-y|^{n-\al}}\ dy \\
&\leq & c\sum_{k}\sum_{\stackrel{Q\in \D}{\ell(Q)=2^k}}  \chi_Q(x)\frac{1}{\ell(Q)^{n-\al}}\int_{|x-y|\leq \ell(Q)} f(y) \ dy  \\
&\leq& c\sum_{Q\in \D} \chi_Q(x) \frac{|Q|^{\al/n}}{|Q|}\int_{3Q} f \ dy
\eey
where the last inequality holds because  if  $x\in Q$,  then $B(x,\ell(Q))\subseteq 3Q$.

One immediately gets then
$$\intRn I_\al (f\sigma)g u \ dx \leq c \, \sum_{\D}\frac{|Q|^{\al/n}}{|Q|}\int_{3Q} f\sigma\, dx  \int_{Q} g u\, dx.
$$
The next crucial step is to pass to a more convenient sum where the
family of dyadic cubes is replaced by an appropriate subset formed
by a family of Calder\'on-Zygmund dyadic cubes. We combine ideas
from the work of  Sawyer and Wheeden in \cite[pp. 824-829]{SW},
together with some techniques from \cite{CP3} (see also \cite{CP2}).

Fix $a>2^n$. Since $g$ is bounded with compact support, for each $k\in \Z$,  one can construct a collection  $\{Q_{k,j}\}_j$ of pairwise disjoint maximal dyadic cubes
(maximal with respect to inclusion)  with the property that
$$
a^k < \frac{1}{|Q_{k,j}|}\int_{Q_{k,j}} gu\,dx.
$$
By maximality  the above also gives
$$
\frac{1}{|Q_{k,j}|}\int_{Q_{k,j}} gu\,dx \leq 2^na^k.
$$
Although the maximal cubes in the whole family  $\{Q_{k,j}\}_{k,j}$ are disjoint in $j$ for each fixed $k$, they  may not be disjoint for different $k$'s.  If we define for each $k$ the collection
$$
{\mathcal C}^k = \{ Q\in \D:  a^k < \frac{1}{|Q_{}|}\int_{Q_{}} gu\,dx \leq a^{k+1}\},
$$
then each dyadic cube $Q$ belongs to only one ${\mathcal C}^k$ or $gu$ vanishes on it. Moreover,
each $Q \in {\mathcal C}^k $ has to be contained in one of the maximal cubes $Q_{k,j_0}$ and  verifies for all $Q_{k,j}$
$$
\frac{1}{|Q_{}|}\int_{Q_{}} gu\,dx \leq a^{k+1} \leq  \frac{a}{|Q_{k,j}|}\int_{Q_{k,j}} gu\,dx.
$$
From these properties and the fact  that for any dyadic cube $Q_0$,
$$
\sum_{Q\in \D, Q\subset Q_0} |Q|^{\al/n} \int_{3Q} f\sigma\, dx
\leq c_\alpha |Q_0|^{\al/n} \int_{3Q_0} f\sigma\, dx,
$$
one easily deduces as in \cite{SW} that
$$
\sum_{\D}\frac{|Q|^{\al/n}}{|Q|}\int_{3Q} f\sigma\, dx  \int_{Q} g u\, dx
\leq  a \, c_{\al}
\sum_{k,j} \frac{|Q_{k,j}|^{\al/n}}{|Q_{k,j}|}\int_{3Q_{k,j}} f\sigma\, dx  \int_{Q_{k,j}} g u\, dx.
$$
Notice also that,
$$[w]_{A_{p_0,q_0}}=\sup_Q \frac{u(Q)}{|Q|}\left(\frac{\sigma(Q)}{|Q|}\right)^{1-\al/n}<\infty,$$
so we can estimate
\begin{eqnarray}
\intRn I_\al (f\sigma)g u \, dx  & \leq & c \, \sum_{k,j}\frac{|Q_{k,j}|^{\al/n}}{|Q_{k,j}|}\int_{3Q_{k,j}}    \!\! f\sigma \, dx \int_{Q_{k,j}}   \!\! g u\, dx  \nonumber \\
&=& c\, \sum_{k,j} \frac{1}{\sigma(5Q_{k,j})^{1-\al/n}}\int_{3Q_{k,j}}   \!\! f \sigma \ dx \ \frac{1}{u(3Q_{k,j})}\int_{Q_{k,j}}   \!\! gu \ dx  \nonumber  \\
& &  \times \, \frac{u(3Q_{k,j})}{|Q_{k,j}|} \left(\frac{\sigma(5Q_{k,j})}{|Q_{k,j}|}\right)^{1-\al/n} |Q_{k,j}| \nonumber  \\
&\leq & c\,[w]_{A_{p_0,q_0}}\sum_{k,j}
\frac{1}{\sigma(5Q_{k,j})^{1-\al/n}}\int_{3Q_{k,j}}  \!\! f \sigma \ dx
\frac{1}{u(3Q_{k,j})}\int_{Q_{k,j}} \!\! \!\!gu \ dx\ |Q_{k,j}|, \label{inamoment}
\end{eqnarray}
where we have set up things to use, in a moment,  certain centered maximal functions.

Before we do so, we need one last property about the
Calder\'on-Zygmund cubes $Q_{k,j}$. We need to pass to a disjoint
collection of sets  $E_{k,j}$ each of which retains a substantial
portion of the mass of the corresponding cube $Q_{k,j}$.

Define the sets
$$E_{k,j}= Q_{k,j}\cap \{x\in R^n:  a^{k}< M^d(gu) \leq a^{k+1}\},$$
where $M^d$ is the dyadic maximal function.
The family $\{E_{k,j}\}_{k,j}$ is pairwise disjoint for all $j$ and $k$.
Moreover, suppose that for some point $x \in Q_{k,j}$ it happens that  $M^d(gu)(x)>a^{k+1}$. By the maximality of $Q_{k,j}$,  this implies that there exist  some dyadic cube $Q$ such that $x\in Q \subset Q_{k,j}$ and  so that
the average of $gu$ over $Q$ is larger than $a^{k+1}$.  It must also hold then that
$M^d(gu\chi_{Q_{k,j}})(x)>a^{k+1}$. But
$$|\{M^d(gu\chi_{Q_{k,j}})>a^{k+1}\} | \leq \frac {1}{a^{k+1}} \int_{Q_{k,j}} gu\,dx\leq \frac {2^n |Q_{k,j}|}{a}.$$
It follows that
$$
|E_{k,j}|\geq (1-\frac{2^n}{a})|Q_{k,j}|.$$
Recalling  now that
$1=u^{\frac{n}{n-\al}}\sigma=u^{\frac{1}{q_0}\frac{n}{n-\al}}\sigma^{\frac{1}{q_0}}$,
we  can use  H\"older's inequality to write
\be \label{eforq}
|Q_{k,j}|\approx |E_{k,j}|=\int_{E_{k,j}}u^{\frac{1}{q_0}\frac{n}{n-\al}}\sigma^{\frac{1}{q_0}}\leq u(E_{k,j})^{1/q_0'}\sigma(E_{k,j})^{1/q_0},
\ee
since
$$\frac{q_0'}{q_0}\frac{n}{n-\al}=1.$$
 With \eqref{eforq} we go back to the string of inequalities to estimate  $\int I_\alpha(f\sigma)\, gu \,dx$. Using the
discrete version of H\"older's inequality, we can estimate in \eqref{inamoment}
 \bey
&\leq & c\,[w]_{A_{p_0,q_0}}\left(\sum_{k,j} \left(\frac{1}{\sigma(5Q_{k,j})^{1-\al/n}}\int_{3Q_{k,j}}f \sigma \ dx\right)^{q_0} \sigma(E_{k,j})\right)^{1/q_0} \\
&&\quad \times \left(\sum_{k,j}\left( \frac{1}{u(3Q_{k,j})}\int_{Q_{k,j}} gu \ dx\right)^{q_0'}u(E_{k,j})\right)^{1/q_0'} \\
&\leq & c\,[w]_{A_{p_0,q_0}}\left(\sum_{k,j} \int_{E_{k,j}}(M^c_{\al,\sigma} f)^{q_0}\sigma \ dx\right)^{1/q_0} \left(\sum_{k,j}\int_{E_{k,j}} (M_u^c g)^{q_0'} u\ dx\right)^{1/q_0'} \\
&\leq & c\,[w]_{A_{p_0,q_0}}\left( \int_{\R^n}(M^c_{\al,\sigma} f)^{q_0}\sigma \ dx\right)^{1/q_0} \left(\int_{\R^n} (M_u^c g)^{q_0'} u\ dx\right)^{1/q_0'} \\
&\leq & c\,[w]_{A_{p_0,q_0}}\left( \int_{\R^n}f^{p_0}\sigma \ dx\right)^{1/p_0} \left(\int_{\R^n} g^{q_0'} u\ dx\right)^{1/q_0'}. \\
\eey
Here we have denoted by $M_u^c = M_{0,u}^c$, the centered maximal function with respect
to the measure $u$. We have also used in the last step Lemma~\ref{WeightFrac},  which gives the
boundedness of $M_u^c$  and $ M^c_{\al,\sigma}$
with operator norms independent of the corresponding measure. We obtain then the desired linear estimate
\begin{equation}\label{initialextraphyp} \|wI_\al f \|_{L^{q_0}}\leq c \, [w]_{A_{p_0,q_0}}\|wf\|_{L^{p_0}}.
\end{equation}
\end{proof}

From this last estimate we can extrapolate (Theorem \ref{ExtThm}) to get,
\begin{equation} \label{applicatextrap}
\|wI_\al f \|_{L^q}\leq
c\,[w]_{A_{p,q}}^{\max\{1,(1-\al/n)p'/q\}}\|wf\|_{L^p}
\end{equation}
for all $1<p<q<\infty$ with $1/p-1/q=\al/n$. Moreover, a  simple duality argument gives then 


%
%
\begin{equation}\label{reducedset}
\|I_\al\|_ {L^p(w^p)\to L^q(w^q) } \leq c [w]_{A_{p,q}}^{\min\{\max(1-\frac{\al}{n}, \frac{p'}{q}), \max( 1, (1-\frac{\al}{n})\frac{p'}{q})\}}.
\end{equation}
This  is sharp for $p'/q\in
(0,1-\al/n]\cup[n/(n-\al),\infty)$.  We obtain the right estimate in the full range of exponents in the next section. The sharpness will be obtained in Section~\ref{examples}.

\section{Proof of the weak-type results and sharp bounds for the full range of exponents}\label{weak-type}

We start with the weak-type version of the extrapolation theorem.

\begin{proof}[Proof of Corollary \ref{ExtWeak}]
Note that Theorem~\ref{ExtThm} does not require $T$ to be linear. We can simply  apply then the result to the operator $T_\lambda f= \lambda \chi_{\{|Tf|>\lambda\}}$.  Fix $\lambda>0$, then
\bey
\|wT_\lambda f \|_{L^{q_0}}&=&\lambda w^{q_0}(\{x : |Tf(x)|>\lambda\})^{1/q_0} \\
&\leq & \|Tf\|_{L^{q_0,\infty}(w^{q_0})} \\
&\leq & c [w]_{A_{p_0,q_0}}^\gamma \|wf\|_{L^{p_0}}, \eey
with constant independent of $\lambda$. Hence by Theorem
\ref{ExtThm} if $w\in A_{p,q}$, $T_\lambda$ maps $L^q(w^q)\ra
L^p(w^p)$ for all $1/p-1/q=1/p_0-1/q_0$ and with bound
$$
\|wT_\lambda f \|_{L^{q}}\leq c\,
[w]_{A_{p,q}}^{\gamma\max\{1,\frac{q_0}{p_0'}\frac{p'}{q}\}}\|fw\|_{L^{p}}.
$$
with $c$ independent of $\lambda$. Hence,
$$\|Tf\|_{L^{q,\infty}(w^q)}=\sup_{\lambda>0} \|wT_\lambda f \|_{L^q}\leq c\, [w]_{A_{p,q}}^{\gamma\max\{1,\frac{q_0}{p_0'}\frac{p'}{q}\}} \|fw\|_{L^{p}}.
$$
\end{proof}

\begin{proof}[Proof of Theorem \ref{ShpWk}]  {\it First Proof (valid for  $p\geq 1$).}

We apply Corollary \ref{ExtWeak} with  $p_0=1$,
$q_0=n/(n-\al)=(n/\al)'$, and $u=w^{q_0}$.

Actually, we are going to prove a better estimate, namely
\be \label{fractF-S.Ineq.} \|I_\al f \|_{L^{q_0,\infty}(u)}\leq c\,
\|f \|_{L^1((Mu)^{1/q_0})} \ee
for any weight $u$.  From this estimate, and since by
\eqref{Aq1} the $A_{1,(n/\al)'}$ condition for $w$ is equivalent to
$$ M(u)\leq [w]_{A_{1,(n/\al)'}} u,$$
we can deduce
$$\|I_\al f \|_{L^{q_0,\infty}(u)}\leq c\,[w]_{A_{1,(n/\al)'}}^{1-\al/n} \|f w\|_{L^1}.$$
The weak extrapolation Corollary \ref{ExtWeak} with
$\gamma={1-\al/n}$ gives the right estimate.

In order to prove \eqref{fractF-S.Ineq.},  we note that
$\|\cdot\|_{L^{q_0,\infty}(u)}$ is equivalent to a norm since
$q_0>1$. Hence, we may use Minkowski's integral inequality as
follows
\be \label{Mink}\|I_\al f\|_{L^{q_0,\infty}(u)} \leq c_q\intRn
|f(y)| \; \||\cdot-y|^{\al-n}\|_{L^{q_0,\infty}(u)} \ dy. \ee
We can finally calculate the inner norm by
\bey
\||\cdot-y|^{\al-n}\|_{L^{q_0,\infty}(w^q)} &=& \sup_{\lambda>0} \lambda u(\{x:|x-y|^{\al-n}>\lambda\})^{1/q_0}\\
&=& (\sup_{t>0} \frac{1}{t^n} u(\{x:|x-y|<t \}))^{1/q_0} \\
&= & c Mu(y)^{1/q_0}.\\
\eey
Once again,  the sharpness of the exponent $1-\al/n$ will be shown with
an example in Section~\ref{examples}.\\

\noindent {\it Second Proof  (valid for  $p> 1$ only).}

We need to recall another characterization of the weak-type
inequality for $I_\alpha$ for two weights. This characterization is
due to Gabidzashvili and Kokilashvili \cite{GK} 
and establishes that for $1<p<q<\infty$, the two-weight weak type
inequality,
\begin{equation}\label{weak2}
\|I_\alpha\|_{L^p(v)\to L^{q,\infty}(u)} <\infty
\end{equation}
hods if and only if
\begin{equation}  \label{e.Glo}
 \sup_Q \left( \int_Q u(x) \
dx\right)^{1/q} \left( \int_{\R^n} ( |Q|^{1/n}
+|x_Q-x|)^{(\alpha-n)p'} v(x)^{1-p'} \ dx\right)^{1/p'} <\infty
\end{equation}
where $x_Q$ denotes the center of the cube $Q$. We will refer to \eqref{e.Glo} as the {\it global testing condition},  given its global character when compared to the {\it local testing conditions} of Sawyer. We will use the notation
$$[u ,v]_{\textup{Glo}(p,q)}=\sup_Q \left( \int_Q u(x) \
dx\right)^{1/q} \left( \int_{\R^n} ( |Q|^{1/n}
+|x_Q-x|)^{(\alpha-n)p'} v(x)^{1-p'} \ dx\right)^{1/p'}.
$$
It follows from the proof  in \cite{GK}  (see also \cite{SW}) that
\begin{equation}  \label{GaKo}
\|I_\alpha\|_{L^p(v)\to L^{q,\infty}(u)} \approx [u
,v]_{\textup{Glo}(p,q)}. \end{equation}

We now need a reverse doubling property satisfied by $w^q$ when
$w\in A_{p,q}$ class (see \cite{SW} for precise definitions).
\begin{lem}\label{l.infty} Let  $w\in A_{p,q}$, then for any cube $ Q$ we have the  estimate
\begin{equation}\label{e.infty}
\frac { \int_Q w ^{q}\;dx } {\int _{2Q} w ^{q} \; dx } \le 1 - c [w
] _{A_{p,q}} ^{-1}
\end{equation}
for an absolute constant $ c$.
\end{lem}
\begin{proof}

Let $ E\subset Q$.  Our goal is to show that
\begin{equation}\label{e.frac<}
\left(\frac {\lvert  E\rvert } {\lvert  Q\rvert } \right) ^{q}
[ w ] _{A_{p,q}}  ^{-1} \le \frac {\int _{E} w ^{q} \; dx } {\int_Q
w ^{q} \; dx}\,.
\end{equation}

Applying this with $ E=Q-\tfrac 12 Q  $ will prove the Lemma.  We can
estimate
\begin{align*}
\frac {\lvert  E\rvert } {\lvert  Q\rvert } & = \frac {\int _E w
\cdot w ^{-1}  } {\lvert  Q\rvert }
\\
& \le \Bigl[ \frac {\int _E  w ^{q} \; dx } {\lvert  Q\rvert }
\Bigr] ^{1/q} \Bigl[ \frac {\int _Q  w ^{-q'} \; dx } {\lvert
Q\rvert } \Bigr] ^{1/q'}
\\
&\le \Bigl[ \frac {\int _E  w ^{q} \; dx } {\lvert  Q\rvert } \Bigr]
^{1/q} \Bigl[ \frac {\int _E  w ^{-p'} \; dx } {\lvert  Q\rvert }
\Bigr] ^{1/p'} & (q'<p')
\\
&= \Bigl[ \frac {\int _E  w ^{q} \; dx } {\int _Q  w ^{q} \; dx  }
\Bigr] ^{1/q} \cdot \Bigl[ \frac {\int _Q  w ^{q} \; dx } {\lvert
Q\rvert } \Bigr] ^{1/q} \Bigl[ \frac {\int _Q  w ^{-p'} \; dx }
{\lvert  Q\rvert } \Bigr] ^{1/p'}
\\
&\le \Bigl[ \frac {\int _E  w ^{q} \; dx } {\int _Q  w ^{q} \; dx  }
\Bigr] ^{1/q} [ w ] _{A_{p,q}} ^{1/q} \,.
\end{align*}
The proof is complete.
\end{proof}

We now claim that in the case $u=w^q$ and $v=w^p$ the constant in the global testing condition and the $A_{p,q}$ constant of $w$ are comparable: 
\begin{equation}\label{e.loc<A}
[w ^{q} ,w ^{p}]_{\textup{Glo}(p,q)} \approx [ w ] _{A_{p,q}}
^{(1-\alpha /n)} \,.
\end{equation}

\begin{proof}[Proof of \eqref{e.loc<A}.]
Observe that $ p' (1-\alpha /n) = 1+ p'/q$. One  of the inequalities in \eqref{e.loc<A} is clear. 
For the other we estimate
\begin{align*}
& \left( \int_Q w(x)^q \ dx\right)^{1/q} \left( \int_{\R^n} (
|Q|^{1/n}
+|x_Q-x|)^{(\alpha-n)p'} w(x)^{p(1-p')} \ dx\right)^{1/p'}\\
 & \le c \left(\int
_{Q} w ^{q}  \right) ^{1/q} \Biggl[ \sum _{j=0} ^{\infty } \lvert  2
^{j} Q\rvert ^{- p' (1-\alpha /n)} \int _{2 ^{j} Q} w ^{-p'} \Biggr]
^{1/p'}
\\
& = c\Biggl[ \sum _{j=0} ^{\infty } \left( \frac {\int _{Q} w ^{q}  }
{\int _{2^j Q} w^q} \right) ^{p'/q} \left( \frac  {\int _{2^j Q} w^q
} {\lvert  2 ^{j}Q\rvert } \right) ^{p'/q} \frac  {\int _{2 ^{j}Q} w
^{p'} } {\lvert  2 ^{j}Q\rvert } \Biggr] ^{1/p'}
\\
& \le c[ w ] _{A_{p,q}} ^{1/q} \Biggl[ \sum _{j=0} ^{\infty } \left(
\frac {\int _{Q} w ^{q}  } {\int _{2^j Q} w^q } \right) ^{p'/q}
\Biggr] ^{1/p'}
\\
& \le c[ w ] _{A_{p,q}} ^{1/q} \Biggl[ \sum _{j=0} ^{\infty }
(1-c [ w ] _{A_{p,q}} ^{-1} ) ^{p'j/q} \Biggr] ^{1/p'}
\\
& \le c[w] _{A_{p,q}} ^{1-\alpha /n} \,.
\end{align*}
Note that the next to last line follows from \eqref{e.infty} and an
immediate inductive argument. In the last line, we just use the
equality $1/q+1/p'=1-\alpha /n $.
\end{proof}

To conclude the second proof of Theorem \ref{ShpWk} we use
\eqref{GaKo}
\begin{equation*}  \label{GaKo2}
\|I_\alpha\|_{L^p(w^p)\to L^{q,\infty}(w^q)} \approx [w^p
,w^q]_{\textup{Glo}(p,q)} \approx [w] _{A_{p,q}} ^{1-\alpha /n}\,.
\end{equation*}

\end{proof}


We conclude this section by verifying that \eqref{strongweak} and
\eqref{weakest} yield Theorem~\ref{fullrange}. Indeed
\begin{equation}\label{strongweak3}
\|I_\alpha\|_{L^p(w^p)\to L^q(w^q)} \approx
\|I_\alpha\|_{L^p(w^p)\to L^{q,\infty}(w^q)}+
\|I_\alpha\|_{L^{q'}(w^{-q'})\to L^{p',\infty}(w^{-p'})}
$$
$$
\approx [w]_{A_{p,q}}^{1-\frac{\al}{n}} +
[w^{-1}]_{A_{q',p'}}^{1-\frac{\al}{n}} \approx
[w]_{A_{p,q}}^{(1-\frac{\al}{n})\max\{1,\frac{p'}{q}\}}
\end{equation}
since $[w^{-1}]_{A_{q',p'}}=[w]_{A_{p,q}}^{p'/q}$ and since
$[w]_{A_{p,q}}\geq1$.

\section{Proof of the sharp bounds for the fractional maximal function}\label{maximalresults}

\begin{proof}[Proof of Theorem \ref{ShpFracMax}] First notice that $M_\al \approx M_\al^c$ where $M_\al^c$ is the centered version.
Let $x\in \R^n$, $Q$ a cube centered at $x$, $u=w^q$,
$\sigma=w^{-p'}$ and $r=1+q/p'$.  Noticing that
$p'/q(1-\al/n)=r'/q$, we proceed as in \cite{Ler} to obtain
\bey
\frac{1}{|Q|^{1-\al/n}}\int_Q |f| \ dy &\leq & 3^{nr'/q}[w]_{A_{p,q}}^{p'/q(1-\al/n)} \left(\frac{|Q|}{u(Q)}\right)^{p'/q(1-\al/n)}\frac{1}{\sigma(3Q)^{1-\al/n}}\int_Q \frac{|f|}{\sigma} \sigma \ dy \\
&\leq&c\,[w]_{A_{p,q}}^{p'/q(1-\al/n)} \left(\frac{1}{u(Q)}\int_Q
M^c_{\al,\sigma}(f/\sigma)^{q/r'} \ dy\right)^{r'/q}. \eey
Taking the supremum over all cubes centered at $x$ we have the
pointwise estimate
$$
M_\al^c f(x) \leq c\,[w]_{A_{p,q}}^{p'/q(1-\al/n)} M_u^c\{M^c_{\al,\sigma}(f/\sigma)^{q/r'}u^{-1}\}(x)^{r'/q}.
$$
Using the fact that $M_u:L^{r'}(u)\ra L^{r'}(u)$ with operator norm
independent of $u$ combined with Lemma \ref{WeightFrac}, we get
\bey
 \|w\,M_\al f \|_{L^q}&\leq& c\, \|M_\al^cf\|_{L^q(u)} \\
 &\leq& c\,[w]_{A_{p,q}}^{p'/q(1-\al/n)} \|M_u^c\{M^c_{\al,\sigma}(f/\sigma)^{q/r'}u^{-1}\}\|_{L^{r'}(u)}^{r'/q} \\
&\leq& c\,[w]_{A_{p,q}}^{p'/q(1-\al/n)}\|f w\|_{L^p} ,\eey
which is the desired estimate.
\end{proof}

\section{Examples}\label{examples}

We will use the power weights considered in \cite{Buck} to show that  Theorems~\ref{ShpWk},
\ref{fullrange},  and \ref{ShpFracMax} are sharp.

Suppose
again $0<\al<n$ with
$$\frac1p-\frac1q=\frac{\al}{n}.$$
Let $w_\delta(x)=|x|^{(n-\delta)/p'}$ so that
$w_\delta\in A_{p,q}$, with
$$[w_\delta]_{A_{p,q}}=[w_\delta^q]_{A_{1+q/p'}}\approx \delta^{-q/p'}.$$
Then, if $f_\delta(x)= |x|^{\delta-n}\chi_{B}$, where $B$ is the
unit ball in $\R^n$, we have
$$\|w_\delta f_\delta  \|_{L^p}\approx \delta^{-1/p}.$$
For $x\in B$,
 $$M_\al f_\delta(x)\geq \frac{C}{|x|^{n-\al}}\int_{B(0,|x|)} |f_\delta(y)| \ dy \approx \frac{|x|^{\delta-n+\al}}{\delta},$$
and so we have
\bey \intRn w_\delta^q M_\al f_\delta(x)^q  \ dx \geq
\delta^{-q}\int_B |x|^{(\delta-n+\al)q}|x|^{(n-\delta)\frac{q}{p'}}
\ dx\approx \delta^{-q-1} .\eey
It follows that
\be \label{delta}\delta^{-1-1/q} \leq c\, \|w_\delta Mf_\delta
\|_{L^q}\leq
c\,[w_\delta]_{A_{p,q}}^{\frac{p'}{q}(1-\frac{\al}{n})}\|w_\delta
f_\delta \|_{L^p}\approx
\delta^{-(1-\frac{\al}{n})}\delta^{-1/p}=\delta^{-1-1/q},\ee
 showing Theorem
\ref{ShpFracMax} is sharp.

Next we show that the same example can be used to show that  the exponent in Theorem \ref{fullrange} is
sharp.  Assume first that $p'/q  \geq 1$
We simply observe that, pointwise,
$$M_\alpha \leq C I_\alpha$$
for some universal constant $C$. Then
using the same
$w_\delta$ and $f_\delta$ as above and the estimate in Theorem~\ref{fullrange}
 we arrive at the estimate in equation \eqref{delta} with $M_\al$
replaced by $I_\al$, showing sharpness.  The case when
$p'/q $
immediately follows by the duality arguments described after the proof of Theorem~\ref{ShpFracInt}.

Finally, we show that the exponent $1-\al/n$ in the estimate
\be \label{weakest1} \|I_\al f\|_{L^{q,\infty}(w^q)}\leq
c\,[w]_{A_{p,q}}^{1-\al/n} \|fw\|_{L^p}\ee
from Theorem \ref{ShpWk} is sharp for $p\geq1$.

By \eqref{property1}
\be \label{weakest2} \|I_\al f\|_{L^{q,\infty}(w^q)}\leq
c\,[w^q]_{A_{1+q/p'}}^{1-\al/n} \|fw\|_{L^p},\ee
and if we let $u=w^q$,
\be \label{weakest4} \|I_\al f\|_{L^{q,\infty}(u)}\leq
c\,[u]_{A_{1+q/p'}}^{1-\al/n} \|f\|_{L^{p}(u^{p/q})}.\ee
Assume now that $u\in A_1$. Then \eqref{weakest4} yields
\be \label{weakest5} \|I_\al f\|_{L^{q,\infty}(u)}\leq
c\,[u]_{A_{1}}^{1-\al/n} \|f\|_{L^{p}(u^{p/q})}.\ee
Since $\frac{p}{q}=1-\frac{p\al}{n}$,  this is equivalent to
\be \label{weakest6} \|I_\al
(u^{\frac{\al}{n}}f)\|_{L^{q,\infty}(u)}\leq c\,[u]_{A_1}^{1-\al/n}
\|f\|_{L^{p}(u)}.
\ee
We now prove that \eqref{weakest6} is sharp.  Let
$$u(x) = |x|^{\de-n}$$
with $0<\de<1$.  Then standard computations shows that
\be \label{I}
[u]_{A_1} \approx \frac{1}{\de }
\ee
Consider the function $f = \chi_{B}$, where $B$ is again the unit ball, we can compute its norm to be
\be \label{II}
\|f\|_{L^{p}(u)}= u(B)^{1/p}= c\left(\frac{1}{\de
}\right)^{1/p}.
\ee
Let $0<x_{\de}<1$ be a parameter whose value will be chosen soon.  We have
\bey \|I_\al
(u^{\al/n}f)\|_{L^{q,\infty}(u)} &\geq& \sup_{\lambda>0}\lambda
\left(u\{ |x|<x_{\de}: \int_B \frac{|y|^{(\de-1)\al/n}}{|x-y|^{1-\alpha/n}}dy>\lambda
\}\right)^{1/q} \\
&\geq& \sup_{\lambda>0}\lambda \left(u\{ |x|<x_{\de}: \int_{B\backslash B(0,|x|)} \frac{|y|^{(\de-1)\al/n}}{|x-y|^{1-\alpha/n}}dy>\lambda \}\right)^{1/q} \\
&\geq& \sup_{\lambda>0}\lambda \left(u\{ |x|<x_{\de}: \int_{B\backslash B(0,|x|)}
\frac{|y|^{(\de-1)\al/n}}{(2|y|)^{1-\alpha/n}}dy>\lambda \}\right)^{1/q} \\
&=&\sup_{\lambda>0}\lambda \left(u\{ |x|<x_{\de}:
\frac{c_{\al,n}}{\de}(1-|x|^{\de\al/n})
>\lambda \}\right)^{1/q} \\
&\geq& \frac{c_{\al,n}}{2\de} \left(u\{ |x|<x_{\de}:
\frac{c_{\al,n}}{\de}(1-|x|^{\de\al/n})
>\frac{c_{\al,n}}{2\de} \}\right)^{1/q}\\
&=&\frac{c_{\al,n}}{2\de} u( B(0,x_\delta))^{1/q}.
\eey
if $x_{\de}= (\frac12)^{n/\al\de}$.  It now follows that for $0<\de<1$,
\be \label{III}
\|I_\al
(u^{\al/n}f)\|_{L^{q,\infty}(u)} \geq \frac{c}{\de}
\left(\frac{x_{\de}^{\de}}{\de}\right)^{1/q}= c\frac{1}{\de}
\left(\frac{1}{\de}\right)^{1/q}.
\ee
Finally, combining \eqref{I},  \eqref{II},  \eqref{III}, and using that
$\frac{1}{q}-\frac{1}{p}=-\frac{\al}{n}$, we have that \eqref{weakest5} is
sharp.

\section{Proof of the Sobolev-type estimate} \label{application}

\begin{proof}[Proof of Theorem \ref{weakstrong}] Since $|f(x)|\leq c I_1(|\nabla f|)(x)$ we can use Theorem \ref{ShpWk} to obtain
\be \label{Wkest} \| f \|_{L^{q,\infty}(w^q)}\leq c
[w]_{A_{p,q}}^{1/n'}\|\nabla f w\|_{L^p}.\ee From this weak-type estimate we can pass to a strong one
with the  procedure that follows.  We use the so-called truncation method from \cite{LN}.

Given a non-negative function $g$ and
$\lambda>0$ we define its truncation about $\lambda$, $\tau_\lambda
g$, to be
$$\tau_\lambda g(x)=\min\{g,2\lambda\}-\min\{g, \lambda\}=\left\{\begin{array}{lc} 0 & g(x)\leq \lambda \\ g(x) -\lambda & \lambda < g(x) \leq 2\lambda \\ \lambda & g(x)> 2\lambda \end{array}\right. .$$
A well-know fact about Lipschitz functions is that they are
preserved by absolute values and truncations.  Define $\Omega_k=\{ x
: 2^k<|f(x)|\leq 2^{k+1}\}$ and let $u=w^q$. Then, \bey
\left(\intRn (|f(x)|w(x))^q \ dx \right)^{1/q}&\leq & \left(\sum_k \int_{\{2^{k+1}<|f(x)|\leq 2^{k+2}\}} |f(x)|^q u(x)\ dx\right)^{1/q} \\
&\leq & c\left(\sum_k 2^{kq} u(\Omega_{k+1})\right)^{1/q} \\
&\leq & c\left(\sum_k 2^{kp} u(\Omega_{k+1})^{p/q}\right)^{1/p}.
\eey Notice that if $x\in \Omega_{k+1}$, then
$\tau_{2^k}|f|(x)=2^k>2^{k-1}$ and hence
$$\Omega_{k+1}\subseteq \{x : \tau_{2^k}|f|(x)>2^{k-1}\}.$$
Furthermore, notice that $|\nabla \tau_{2^k}(|f|)|= |\nabla |f| |
\chi_{\Omega_{k}}\leq |\nabla f |\chi_{\Omega_{k}}$, a.e..
Continuing and using the weak-type estimate \eqref{Wkest} we have
 \bey
 \|f\|_{L^q(w^q)}
&\leq & c\left(\sum_k (2^k u(\{ x: \tau_{2^k}|f|(x)>2^{k-1}\})^{1/q})^p\right)^{1/p} \\
&\leq & c\,[w]_{A_{p,q}}^{1/n'} \left(\sum_k \int_{\Omega_k} (|\nabla\tau_{2^k}|f|(x)|w(x))^p \ dx\right)^{1/p}\\
&\leq & c\,[w]_{A_{p,q}}^{1/n'}\left(\intRn (|\nabla f(x)|w(x))^p \
dx\right)^{1/p},
\eey
since $p<q$ and  the sets $\Omega_k$ are disjoint. This
finishes the proof of the theorem.
 \end{proof}


\begin{thebibliography}{9}


\bibitem{AsIwSa} K. Astala, T. Iwaniec, and E. Saksman, \emph{Beltrami operators in the plane}, Duke Math. J., {\bf 107} (2001),  no. 1,  27-56.

\bibitem{Buck} S. Buckley, \emph{Estimates for operator norms on weighted spaces and reverse Jensen inequalities}, Trans. Amer. Math. Soc., {\bf 340} (1993), no. 1, 253-272.

\bibitem{CF} R. Coifman and C. Fefferman, \emph{Weighted norm inequalities for maximal functions and singular integrals}, Studia Math.., {\bf 51} (1974), 241-250.

\bibitem{DGPP} O. Dragi\u{c}evi\'c, L. Grafakos, C. Pereyra, and S. Petermichl, \emph{Extrapolation and sharp norm estimates for classical operators on weighted Lebegue spaces} Publ. Math {\bf 49} (2005), no. 1, 73-91.

\bibitem{GK}  M. Gabidzashvili and V. Kokilashvili, {\it Two weight weak type inequalities for fractional
type integrals}, Ceskoslovenska Akademie Ved., {\bf 45} (1989), 1-11.

\bibitem{GCRdF} J.\ Garc\'\i a-Cuerva and J.L.\ Rubio de Francia, {\em Weighted Norm Inequalities and Related Topics}, North Holland Math. Studies 116, North Holland, Amsterdam, 1985.


\bibitem{GrCF} L. Grafakos, \emph{Classical Fourier Analysis}, Springer-Verlag, Graduate Texts in Mathematics \textbf{249},
Second Edition 2008.

\bibitem{GrMF} L. Grafakos, \emph{Modern Fourier Analysis}, Springer-Verlag, Graduate Texts in Mathematics \textbf{250}, Second
Edition 2008.


\bibitem{GM} L. Grafakos and J.M. Martell, \emph{Extrapolation of weighted norm inequalities for multivariable operators and applications}, J. Geom. Anal. {\bf 14} (2004), 19-46.

\bibitem{HMS} E. Harboure, R. Mac\'ias, and C. Segovia, \emph{Extrapolation results for classes of weights}, Amer. J. Math., {\bf 110} (1988), 383-397.

\bibitem{Haj} P. Hajlasz, {\em Sobolev inequalities, truncation method, and
John domains,}  Papers on analysis: A volume dedicated to Olli
Martio on the occasion of his 60th birthday,  Report Univ. Jyväskylä
Dep. Math. Stat., 83, p. 109--126; Univ. Jyväskylä, Jyväskylä, 2001.

\bibitem{HMW} R. Hunt, B. Muckenhoupt, and R. Wheeden, \emph{Weighted norm inequalities for the conjugate function and Hilbert transform}, Trans. Amer. Math. Soc., {\bf 176} (1973), 227-251.

\bibitem{Ler} A. Lerner, \emph{An elementary approach to several results on the Hardy-Littlewood maximal operator},  Proc. AMS.
{\bf 136} (2008),  no. 8, 2829-2833.

\bibitem{LOP1} A. Lerner, S. Ombrosi, and C. P\'erez, \emph{Sharp $A_1$ bounds for Calder\'on-Zygmund  operators and the relationship with a problem of Muckenhoupt and Wheeden}, Int.\  Math.\  Res.\  Notices (2008) Vol. 2008.

\bibitem{LOP2} A. Lerner, S. Ombrosi and C. P\'erez, {\em $A_1$ bounds for Calder\'on-Zygmund operators related to a problem of Muckenhoupt and
Wheeden.} Math. Res. Lett. {\bf 16}, to appear (2009).

\bibitem{LN} R. Long and F. Nie, \emph{Weighted Sobolev inequalities and eigenvalue estimate of Schr\"odinger operators}, Lecture Notes in Math., 1494 (1990), 131-141.

\bibitem{Mu} B. Muckenhoupt, \emph{Weighted norm inequalities for the Hardy maximal function}, Trans. Amer. Math. Soc., {\bf 165} (1972), 207-226.

\bibitem{MW} B. Muckenhoupt and R. Wheeden, \emph{Weighted norm inequalities for fractional integrals}, Trans. Amer. Math. Soc., {\bf 192} (1974), 261-274.


\bibitem{CP2} C. P\'erez, \emph{Two weight inequalities for potential and fractional type maximal operators}, Indiana Univ. Math. J. {\bf 43}, No. 2 (1994), 663-683.

\bibitem{CP3} C. P\'erez, \emph{Sharp $L\sp p$-weighted Sobolev inequalities}, Ann. Inst. Fourier {\bf 45} (1995),  no. 3, 1-16.

\bibitem{Pet1} S. Petermichl, \emph{The sharp bound for the Hilbert transform in weighted Lebesgue spaces in terms of the classical $A_p$ characteristic}, Amer. J. Math.  {\bf 129}  (2007),  no. 5, 1355--1375.

\bibitem{Pet2} S. Petermichl, \emph{The sharp weighted bound for the Riesz transforms}, The sharp weighted bound for the Riesz transforms.  Proc. Amer. Math. Soc.  {\bf 136}  (2008),  no. 4, 1237--1249.

\bibitem{PetVol} S. Petermichl and A. Volberg, \emph{Heating of the Ahlfors-Beurling operator: weakly quasiregular maps on the plane are quasiregular}, Duke Math. J.,  {\bf 112}  (2002),  no. 2, 281-305.

\bibitem{Rub} J. L. Rubio de Francia, \emph{Factorization theory and $A_{p}$ weights}, Amer. J. Math. 106 (1984), no. 3, 533--547.

\bibitem{Saw1} E. Sawyer,  \emph{A characterization of a two-weight norm inequality for maximal operators}, Studia Math. {\bf 75} (1982), 1-11.

\bibitem{Saw11} E. Sawyer \emph{A two weight weak type inequality for fractional integrals}, Trans. Amer. Math. Soc. {\bf 281} (1984), no. 1, 339--345.

\bibitem{Saw2} E. Sawyer, \emph{A characterization of a two-weight norm inequality for fractional and poisson integrals}, Trans. Amer. Math. Soc. {\bf 308} (1988),  no. 2, 533-545.

\bibitem{SW} E. Sawyer and R. Wheeden, \emph{Weighted inequlities for fractional integrals on euclidean and homogeneous spaces}, Amer. J. Math. {\bf 114} (1992), 813-874.

\bibitem{St} E. Stein, \emph{Singular Integrals and Differentiability Properties of Functions}, Princton Univ. Press, Princeton, New Jersey, 1970.



\end{thebibliography}
\end{document}